\documentclass{article}

\usepackage{amsfonts,amsthm,amsmath}
\usepackage[dvipdfmx]{color}
\numberwithin{equation}{section}

\usepackage{amssymb}

\theoremstyle{plain}
\newtheorem{thm}{Theorem}[section]
\newtheorem{prop}{Proposition}[section]
\newtheorem{lem}{Lemma}[section]

\theoremstyle{definition}

\newtheorem{rem}{Remark}[section]

\begin{document}


\title{On the average hitting times of the squares of cycles}
\author{Yoshiaki Doi\thanks{Sakata City Daini Junior High School, Shinbashi, Sakata 998--0864, Japan, y.doi6532@gmail.com
}, 
Norio Konno\thanks{Department of Applied Mathematics, Faculty of Engineering, Yokohama National University, Hodogaya, Yokohama, 240-8501, Japan, konno-norio-bt@ynu.ac.jp
},
Tomoki Nakamigawa\thanks{Department of Information Science
Shonan Institute of Technology, Tsujido-Nishikaigan, Fujisawa, 251-8511, Japan, nakami@info.shonan-it.ac.jp
}, 
Tadashi Sakuma\thanks{Faculty of Science, 
Yamagata University, Kojirakawa, Yamagata, 990--8560, Japan, sakuma@sci.kj.yamagata-u.ac.jp
}, \\
Etsuo Segawa\thanks{Graduate School of Education Center and Graduate School of Environment Information Sciences, Yokohama National University, Hodogaya, Yokohama, 240-8501, Japan, segawa-etsuo-tb@ynu.ac.jp
},
Hidehiro Shinohara\thanks{Institute for Excellence in Higher Education, 
Tohoku University, Aoba, Sendai, 980-8576, Japan, shino.setsystem@gmail.com
},
Shunya Tamura\thanks{Graduate School of Science and Engineering, Yokohama National University, Hodogaya, Yokohama, 240-8501, Japan, tamura-shunya-kj@ynu.jp
}, 
Yuuho Tanaka\thanks{Graduate School of Information Sciences, Tohoku University Aoba, Sendai, 980-8579, Japan, tanaka.yuuho.t4@dc.tohoku.ac.jp
}, 
Kosuke Toyota\thanks{Graduate School of Science and Engineering (Science), 
Yamagata University, Kojirakawa, Yamagata, 990--8560, Japan, s201001m@st.yamagata-u.ac.jp
}. 
}
\date{}
\maketitle

\begin{abstract}
The exact formula for the average hitting time (HT, as an abbreviation) of simple random 
walks from one vertex to any other vertex on the square $C^2_N$ of an $N$-vertex
cycle graph $C_N$ was given by N. Chair [\textit{Journal of Statistical Physics},
\textbf{154} (2014) 1177-1190]. In that paper, the author gives the expression for
the even $N$ case and the expression for the odd $N$ case separately. In this paper,
by using an elementary method different from Chair (2014), we give a much simpler
single formula for the HT's of simple random walks on $C^2_N$. Our proof is considerably
short and fully combinatorial, in particular, has no-need of any spectral graph
theoretical arguments. Not only the formula itself but also intermediate results
through the process of our proof describe clear relations between the HT's of simple
random walks on $C^2_N$ and the Fibonacci numbers.
\end{abstract}
\noindent
Keywords: simple random walk, hitting time, the square of a cycle, Fibonacci number, Kirchhoff index

\section{Introduction}

A simple random walk on a graph $G$ is a discrete stochastic model
such that a random walker on a vertex $u\in V(G)$ moves to a vertex $v$
adjacent to the vertex $u$ at the next step with the probability of
$\frac{1}{\deg_G(u)}$, where $\deg_G(u)$ denotes the degree of the vertex $u$
of $G$. The number of steps required until the random walker starting from
a vertex $p$ of $G$ will first arrive at a vertex $q$ of $G$ is called
{\it the hitting time} from $p$ to $q$ of the simple random walk on $G$.
{\it The average hitting time} (HT, as an abbreviation) from $p$ to $q$ on $G$ 
which is denoted by $h(G;p,q)$, means the expected value of the hitting times
from $p$ to $q$ of simple random walks on $G$. The exact formula for
the average hitting time from one vertex to any other vertex is far from
available in general. For some very special graph classes with high symmetries,
it may be possible to obtain such exact formulas. 
For example, as far as the authors know, only in the case of $i\in\{1,2\}$,
exact formulas for the $i$-th power $C^i_N$ of an $N$-vertex cycle graph $C_N$
are obtained. Furthermore, it was not until 2014 that such a formula for
the case of $i=2$ was given by N. Chair \cite{Chair}.
In that paper, by calculating the exact value of Wu's formula \cite{Wu} of
the effective resistances in terms of the eigenvalues and the eigenvectors of 
the Laplacian matrices of $C^2_N$, the author gives the expression for the even
$N$ case and the expression for the odd $N$ case separately.
Let us define $V(C^2_{N}):= \mathbb{Z}_{N}$ and
$E(C^2_{N}):= \{ \{u,v\} \mid u-v=\pm1,\pm2 \}$. 
Let us denote by $h_N(k,l)$ the average hitting time for simple random walks
from the vertex $k$ to the vertex $l$ of $C^2_N$.
Note that $h_N(0,0)=0$.
Since the dihedral group $D_N$ acts on $C^2_N$, we have that 
$\forall k,l \in\mathbb{Z}_N,h_N(0,l)=h_N(k,k+l)=h_N(k+l,k)$.
Let $F_i$ denote the $i$-th Fibonacci number.
\begin{thm}[Chair, 2014 \cite{Chair}]
The exact formula for the HT's of simple random walks on $C^2_N(N\geq 5)$ is as follows:
\begin{enumerate}
\item For the case of $N=2n$,  
\[
h_N(0,l)=\frac{2}{5}l\left(N-l\right)+\left(-1\right)^{l+1}\frac{2N}{\sqrt{5}}F_l^2\left(\frac{1+\left(\frac{3-\sqrt{5}}{2}\right)^N}{1-\left(\frac{3-\sqrt{5}}{2}\right)^N}\right)+\left(-1\right)^l\frac{2N}{5}F_{2l}.
\]
\item For the case of $N=2n+1$, 
\[
h_N(0,l)=\frac{2}{5}l\left(N-l\right)+\left(-1\right)^{l+1}\frac{2N}{\sqrt{5}}F_l^2\left(\frac{1-\left(\frac{3-\sqrt{5}}{2}\right)^N}{1+(\frac{3-\sqrt{5}}{2})^N}\right)+\left(-1\right)^l\frac{2N}{5}F_{2l}.
\]
\end{enumerate}
\end{thm}

In this paper, by using an elementary method different from 
Chair \cite{Chair}, we give a simpler single formula for HT's of simple random
walks on $C^2_N$, as follows: 
\begin{thm}\label{main}
The exact formula for the HT's of simple random walks on $C^2_N (N \geq 5)$ is, 
\[
h_N(0,l)=\frac{2}{5}\left(l\left(N-l\right)+2N\frac{F_{l}F_{N-l}}{F_{N}}\right).
\] 
\end{thm}
Our proof is considerably short and fully combinatorial.
In particular, it has no-need of any spectral graph theoretical arguments.
Not only the formula itself but also intermediate results through the process of
our proof describe clear relations between the HT's of simple random walks on
$C^2_N$ and the Fibonacci numbers. 

\section{Preliminary}
In this section, we prepare some well-known formulas on Fibonacci numbers, 
which will be used throughout this paper. 
For proofs of these formulas, please refer to the appropriate
literature on Fibonacci numbers (e.g., \cite{Thomas,Koshy}).

\begin{align}
&F_{n+1}F_{n-1}-F_n^2=(-1)^n. \label{t1} \\
&F_{2n}-3F_{2(n-1)}+F_{2(n-2)}=0. \label{t13} \\
&F_{2n+1}=F_{n+1}^2+F_n^2. \label{t3} \\
&F_{m-n}=(-1)^{n+1}F_{m-1}F_n+(-1)^nF_mF_{n-1}. \label{t5} \\
&\frac{1}{F_{2n-1}F_{2n+1}}=\frac{F_{2n+2}}{F_{2n+1}}-\frac{F_{2n}}{F_{2n-1}}. \label{t7} \\
&\sum_{k=1}^l{(-1)^{l-k}F_{2k-1}}=F_l^{2}. \label{t10} \\
&F_{l}^{2}=\frac{1}{5}(F_{2l-1}+F_{2l+1})+\frac{2}{5}(-1)^{l-1}. \label{t11}
\end{align}

\section{Proof of Theorem  \ref{main}}
Let $L$ denote the Laplacian matrix of $C^2_N$ and 
let $L^{\prime}$ denote the matrix obtained from $L$ by deleting its first row and column. 
Let $\vec{h}$ be the column vector whose $i$-th entry is $h_N(0,i)$. 
Let $\vec{1}$ be a column vector of proper dimension whose entries are all $1$.
For the case of a random walk on $C^2_N$, a random walker moves, 
with the probability of $\frac{1}{4}$, to an arbitrary vertex adjacent
to the vertex where the walker is. Hence we have that $\forall l \in\mathbb{Z}_N$, 
$h_N(0,l)=\frac{1}{4}\left((1+h_N(-2,l))+(1+h_N(-1,l))+(1+h_N(1,l))+(1+h_N(2,l))\right)$,
and hence $\forall l \in\mathbb{Z}_N, -h_N(0,l-2)-h_N(0,l-1)+4h_N(0,l)-h_N(0,l+1)-h_N(0,l+2)=4$. 
Thus, we have $L' \vec{h}=4\vec{1}$. Combining this with the fact that
$h_N(0,l)=h_N(0,N-l)$ for all $l \in \mathbb{Z}_{N}$, we can halve the number of variables of
our problem. Let $L'(i, j)$ denote the $(i, j)$-th entry of $L'$. 
Let $H_N$ be the $\lfloor{N/2}\rfloor\times\lfloor{N/2}\rfloor$-matrix whose $(i,j)$-th entry is 
$L'(i, j)$ if $j=\frac{N}{2} (N=2n)$, and $L'(i, j)+L'(i, N-j)$ otherwise. 
Let $\vec{x}$ be an $\lfloor{N/2}\rfloor$-dimensional variable column vector whose $i$-th entry is
corresponding to $h_N(0,i)$. Then we have $H_N \vec{x}=$$4\vec{1}$.
For our problem, it suffices to solve this matrix equation. 
As a technical issue, in the case of $N=2n$, our matrix $H_{2n}$ is not symmetric.
However, if we multiply the last row of this matrix by $\frac{1}{2}$, then the resultant matrix
$H'_{2n}$ will be symmetric. Hence, only in the case of $N=2n$, we redefine
$H_{2n}:=H'_{2n}$ and change our matrix equation to $H_{2n} \vec{x}=$${}^{t}(4,\ldots,4,2)$.

Let $U_N$ be the upper triangular matrix whose $(i,j)$-th entry is $1$ if $j \geq i$.

Let $W_N$ be the following matrix:
\[
W_{N}:=
\begin{bmatrix}
1&0&0&\cdots&0&0\\
\frac{F_1}{F_3}&1&0&\cdots&0&0\\
0&\frac{F_3}{F_5}&1&\cdots&0&0\\
\vdots&\vdots&\vdots&\ddots&\vdots&\vdots \\
0&0&0&\cdots&1&0\\
0&0&0&\cdots&\frac{F_{2\left \lfloor \frac{N}{2} \right \rfloor-3}}{F_{2\left \lfloor \frac{N}{2} \right \rfloor-1}}&1\\
\end{bmatrix}. 
\]
Then its inverse $W_N^{-1}$ is as follows:
\[
W_N^{-1}=
\begin{bmatrix}
1&0&0&\cdots&0&0\\

-\frac{F_{1}}{F_{3}}&1&0&\cdots&0&0\\

\frac{F_{1}}{F_{5}}&-\frac{F_{3}}{F_{5}}&1&\cdots&0&0\\

\vdots&\vdots&\vdots&\ddots&\vdots&\vdots\\

(-1)^{\left \lfloor \frac{N}{2} \right \rfloor-2}\frac{F_{1}}{F_{2\left \lfloor \frac{N}{2} \right \rfloor-3}}&(-1)^{\left \lfloor \frac{N}{2} \right \rfloor-3}\frac{F_{3}}{F_{2\left \lfloor \frac{N}{2} \right \rfloor-3}}&(-1)^{\left \lfloor \frac{N}{2} \right \rfloor-4}\frac{F_{5}}{F_{2\left \lfloor \frac{N}{2} \right \rfloor-3}}&\cdots&1&0\\

(-1)^{\left \lfloor \frac{N}{2} \right \rfloor-1}\frac{F_{1}}{F_{2\left \lfloor \frac{N}{2} \right \rfloor-1}}&(-1)^{\left \lfloor \frac{N}{2} \right \rfloor-2}\frac{F_{3}}{F_{2\left \lfloor \frac{N}{2} \right \rfloor-1}}&(-1)^{\left \lfloor \frac{N}{2} \right \rfloor-3}\frac{F_{5}}{F_{2\left \lfloor \frac{N}{2} \right \rfloor-1}}&\cdots&-\frac{F_{2\left \lfloor \frac{N}{2} \right \rfloor-3}}{F_{2\left \lfloor \frac{N}{2} \right \rfloor-1}}&1\\
\end{bmatrix}.
\]

Let $D_N$ be the following matrix:
\[
D_{N}:=
\begin{bmatrix}
\frac{F_{3}}{F_{1}}&0&0&\cdots&0&0\\
0&\frac{F_{5}}{F_{3}}&0&\cdots&0&0\\
0&0&\frac{F_{7}}{F_{5}}&\cdots&0&0\\
\vdots&\vdots&\vdots&\ddots&\vdots&\vdots\\
0&0&0&\cdots&\frac{F_{2\left \lfloor \frac{N}{2} \right \rfloor-1}}{F_{2\left \lfloor \frac{N}{2} \right \rfloor-3}}&0\\
0&0&0&\cdots&0&\frac{F_{N}}{F_{2\left \lfloor \frac{N}{2} \right \rfloor-1}}\\
\end{bmatrix}. 
\]

Then we have the following:
\begin{thm}\label{decomposition}
$H_N=U^{-1}_N\,W_ND_N{}^tW_N\,{}^tU^{-1}_N$. 
\end{thm}

\noindent
{\it Proof of Theorem \ref{decomposition}.}

Let $S_N$ denote the matrix $U_N^{-1}W_ND_N{}^tW_N{}^tU_N^{-1}$.
By simple matrix-computation, the entries of $S_N$ can be calculated as follows. 
\[
S_N(i,j)=
\begin{cases}
0 &\text{($|i-j|\geq 3$)}, \\
-1 &\text{($|i-j|=2$)}, \\
2-\frac{F_{2i-3}+F_{2i+1}}{F_{2i-1}}  &\text{($i-j=1$)}, \\
2-\frac{F_{2j-3}+F_{2j+1}}{F_{2j-1}}  &\text{($j-i=1,j\neq\left \lfloor \frac{N}{2} \right \rfloor$)}, \\
1-\frac{F_{2\lfloor \frac{N}{2}\rfloor-3}+F_N}{F_{2\lfloor \frac{N}{2}\rfloor-1}} &\text{($(i,j)=(\left \lfloor \frac{N}{2} \right \rfloor-1,\left \lfloor \frac{N}{2} \right \rfloor)$)}, \\
\frac{F_{2i-3}+F_{2i+1}}{F_{2i-1}}+\frac{F_{2i-1}+F_{2i+3}}{F_{2i+1}}-2  &\text{($i=j\not\in \{ 1,\left \lfloor \frac{N}{2} \right \rfloor-1,\left \lfloor \frac{N}{2} \right \rfloor\}$)}, \\
\frac{F_3}{F_1}+\frac{F_1+F_5}{F_3}-2  &\text{($(i,j)=(1,1)$)}, \\
\frac{F_{2\lfloor \frac{N}{2}\rfloor-1}+F_{2\lfloor \frac{N}{2}\rfloor-5}}{F_{2\lfloor \frac{N}{2}\rfloor-3}}+\frac{F_{2\lfloor \frac{N}{2}\rfloor-3}+F_N}{F_{2\lfloor \frac{N}{2}\rfloor-1}}-2 &\text{($(i,j)=(\left \lfloor \frac{N}{2} \right \rfloor-1,\left \lfloor \frac{N}{2} \right \rfloor-1)$)}, \\
\frac{F_{2\lfloor \frac{N}{2}\rfloor-3}+F_N}{F_{2\lfloor \frac{N}{2}\rfloor-1}}  &\text{($(i,j)=(\left \lfloor \frac{N}{2} \right \rfloor,\left \lfloor \frac{N}{2} \right \rfloor)$)}. 
\end{cases}
\]

In the case that $|i-j|=1,j \neq \lfloor \frac{N}{2}\rfloor$, we have
\begin{align*}
  S_N(i,j)&=2-\frac{F_{2l-1}+F_{2l+3}}{F_{2l+1}}\\
          &=2-\frac{3F_{2l+1}}{F_{2l+1}} &\text{(by (\ref{t13}))}\\
&=-1. 
\end{align*}

In the case that $i=j\not\in \{1,\left \lfloor \frac{N}{2} \right \rfloor-1,\left \lfloor \frac{N}{2} \right \rfloor\}$, we have  
\begin{align*}
S_N(i,j)=\frac{F_{2l-1}+F_{2l+3}}{F_{2l+1}}+\frac{F_{2l+1}+F_{2l+5}}{F_{2l+3}}-2&=\frac{3F_{2l+1}}{F_{2l+1}}+\frac{3F_{2l+3}}{F_{2l+3}}-2 &&\text{(by (\ref{t13}))}\\
&=4 . 
\end{align*}

In the case of $S_N(1,1)$, we have 
\[
S_N(1,1)=\frac{F_3}{F_1}+\frac{F_1+F_5}{F_3}-2=2+\frac{1+5}{2}-2=3. 
\]

We will show $S_N(\left \lfloor \frac{N}{2} \right \rfloor-1,\left \lfloor \frac{N}{2} \right \rfloor-1)=H_N(\left \lfloor \frac{N}{2} \right \rfloor-1,\left \lfloor \frac{N}{2} \right \rfloor-1)$. 

In the case that $N$ is odd, 
\begin{align*}
S_N\left(\left \lfloor \frac{N}{2} \right \rfloor-1,\left \lfloor \frac{N}{2} \right \rfloor-1\right)
&=\frac{F_{2\lfloor \frac{N}{2}\rfloor-1}+F_{2\lfloor \frac{N}{2}\rfloor-5}}{F_{2\lfloor \frac{N}{2}\rfloor-3}}+\frac{F_{2\lfloor \frac{N}{2}\rfloor-3}+F_N}{F_{2\lfloor \frac{N}{2}\rfloor-1}}-2\\
&=\frac{F_{2\lfloor \frac{N}{2}\rfloor-1}+F_{2\lfloor \frac{N}{2}\rfloor-5}}{F_{2\lfloor \frac{N}{2}\rfloor-3}}+\frac{F_{2\lfloor \frac{N}{2}\rfloor-3}+F_{2\lfloor \frac{N}{2}\rfloor+1}}{F_{2\lfloor \frac{N}{2}\rfloor-1}}-2\\ 
&=\frac{3F_{2\lfloor \frac{N}{2}\rfloor-3}}{F_{2\lfloor \frac{N}{2}\rfloor-3}}+\frac{3F_{2\lfloor \frac{N}{2}\rfloor-1}}{F_{2\lfloor \frac{N}{2}\rfloor-1}}-2&&\text{(by (\ref{t13}))}\\
&=4. 
\end{align*}

In the case that $N$ is even, 
\begin{align*}
S_N\left(\left \lfloor \frac{N}{2} \right \rfloor-1,\left \lfloor \frac{N}{2} \right \rfloor-1\right)
&=\frac{F_{2\lfloor \frac{N}{2}\rfloor-1}+F_{2\lfloor \frac{N}{2}\rfloor-5}}{F_{2\lfloor \frac{N}{2}\rfloor-3}}+\frac{F_{2\lfloor \frac{N}{2}\rfloor-3}+F_N}{F_{2\lfloor \frac{N}{2}\rfloor-1}}-2\\
&=\frac{F_{2\lfloor \frac{N}{2}\rfloor-1}+F_{2\lfloor \frac{N}{2}\rfloor-5}}{F_{2\lfloor \frac{N}{2}\rfloor-3}}+\frac{F_{2\lfloor \frac{N}{2}\rfloor-3}+F_{2\lfloor \frac{N}{2}\rfloor}}{F_{2\lfloor \frac{N}{2}\rfloor-1}}-2\\
&=\frac{F_{2\lfloor \frac{N}{2}\rfloor-1}+F_{2\lfloor \frac{N}{2}\rfloor-5}}{F_{2\lfloor \frac{N}{2}\rfloor-3}}+\frac{F_{2\lfloor \frac{N}{2}\rfloor-3}+F_{2\lfloor \frac{N}{2}\rfloor+1}-F_{2\lfloor \frac{N}{2}\rfloor-1}}{F_{2\lfloor \frac{N}{2}\rfloor-1}}-2\\
&=\frac{3F_{2\lfloor \frac{N}{2}\rfloor-3}}{F_{2\lfloor \frac{N}{2}\rfloor-3}}+\frac{2F_{2\lfloor \frac{N}{2}\rfloor-1}}{F_{2\lfloor \frac{N}{2}\rfloor-1}}-2&&\text{(by (\ref{t13}))}\\
&=3. 
\end{align*}

We will show $S_N(\lfloor \frac{N}{2} \rfloor-1,\left \lfloor \frac{N}{2} \right \rfloor)=H_N(\left \lfloor \frac{N}{2} \right \rfloor-1,\left \lfloor \frac{N}{2} \right \rfloor)$.

In the case that $N$ is odd, 
\begin{align*}
S_N\left(\left \lfloor \frac{N}{2} \right \rfloor-1,\left \lfloor \frac{N}{2} \right \rfloor\right)
&=1-\frac{F_{2\lfloor \frac{N}{2}\rfloor-3}+F_N}{F_{2\lfloor \frac{N}{2}\rfloor-1}}\\
&=1-\frac{F_{2\lfloor \frac{N}{2}\rfloor-3}+F_{2\lfloor \frac{N}{2}\rfloor+1}}{F_{2\lfloor \frac{N}{2}\rfloor-1}}\\
&=1-\frac{3F_{2\lfloor \frac{N}{2}\rfloor-1}}{F_{2\lfloor \frac{N}{2}\rfloor-1}}&&\text{(by (\ref{t13}))}\\
&=-2.
\end{align*}

In the case that $N$ is even, 
\begin{align*}
S_N\left(\left \lfloor \frac{N}{2} \right \rfloor-1,\left \lfloor \frac{N}{2} \right \rfloor\right)
&=1-\frac{F_{2\lfloor \frac{N}{2}\rfloor-3}+F_N}{F_{2\lfloor \frac{N}{2}\rfloor-1}}\\
&=1-\frac{F_{2\lfloor \frac{N}{2}\rfloor-3}+F_{2\lfloor \frac{N}{2}\rfloor}}{F_{2\lfloor \frac{N}{2}\rfloor-1}}\\
&=1-\frac{F_{2\lfloor \frac{N}{2}\rfloor-3}+F_{2\lfloor \frac{N}{2}\rfloor+1}-F_{2\lfloor \frac{N}{2}\rfloor-1}}{F_{2\lfloor \frac{N}{2}\rfloor-1}}\\
&=1-\frac{2F_{2\lfloor \frac{N}{2}\rfloor-1}}{F_{2\lfloor \frac{N}{2}\rfloor-1}}&&\text{(by (\ref{t13}))}\\
&=-1. 
\end{align*}

Last, we will show $S_N(\left \lfloor \frac{N}{2} \right \rfloor,\left \lfloor \frac{N}{2} \right \rfloor)=H_N(\left \lfloor \frac{N}{2} \right \rfloor,\left \lfloor \frac{N}{2} \right \rfloor)$

In the case that $N$ is odd, 
\begin{align*}
S_N\left(\left \lfloor \frac{N}{2} \right \rfloor,\left \lfloor \frac{N}{2} \right \rfloor\right)
&=\frac{F_{2\left \lfloor \frac{N}{2} \right \rfloor-3}+F_{N}}{F_{2\left \lfloor \frac{N}{2} \right \rfloor-1}}\\
&=\frac{F_{2\left \lfloor \frac{N}{2} \right \rfloor-3}+F_{2\left \lfloor \frac{N}{2} \right \rfloor+1}}{F_{2\left \lfloor \frac{N}{2} \right \rfloor-1}}\\
&=\frac{3F_{2\left \lfloor \frac{N}{2} \right \rfloor-1}}{F_{2\left \lfloor \frac{N}{2} \right \rfloor-1}}&&\text{(by (\ref{t13}))}\\
&=3. 
\end{align*}

In the case that $N$ is even, 
\begin{align*}
S_N\left(\left \lfloor \frac{N}{2} \right \rfloor,\left \lfloor \frac{N}{2} \right \rfloor\right)
&=\frac{F_{2\left \lfloor \frac{N}{2} \right \rfloor-3}+F_{N}}{F_{2\left \lfloor \frac{N}{2} \right \rfloor-1}}\\
&=\frac{F_{2\left \lfloor \frac{N}{2} \right \rfloor-3}+F_{2\left \lfloor \frac{N}{2} \right \rfloor}}{F_{2\left \lfloor \frac{N}{2} \right \rfloor-1}}\\
&=\frac{F_{2\left \lfloor \frac{N}{2} \right \rfloor-3}+F_{2\left \lfloor \frac{N}{2} \right \rfloor+1}-F_{2\left \lfloor \frac{N}{2} \right \rfloor-1}}{F_{2\left \lfloor \frac{N}{2} \right \rfloor-1}}\\
&=\frac{2F_{2\left \lfloor \frac{N}{2} \right \rfloor-1}}{F_{2\left \lfloor \frac{N}{2} \right \rfloor-1}}&&\text{(by (\ref{t13}))}\\
&=2. 
\end{align*}

Therefore, we have $S_N=H_N$. \qed

Now let us define the new variable vector $\vec{y}={}^tU_N^{-1}\vec{x}$.
Then both the matrix equation $H_N \vec{x}=4\vec{1}$ in
the case of $N=2n+1$ and the matrix equation $H_N \vec{x}={}^{t}(4,4,\ldots,2)$
in the case of $N=2n$ can be expressed by the following single matrix equation. 
\begin{equation*}
\vec{y}=
\begin{bmatrix}
y_1\\
y_2\\
\vdots\\
y_l\\
\vdots\\
y_{\left \lfloor \frac{N}{2} \right \rfloor}\\
\end{bmatrix}
:=
\begin{bmatrix}
x_1\\
x_2-x_1\\
\vdots\\
x_l-x_{l-1}\\
\vdots\\
x_{\left \lfloor \frac{N}{2} \right \rfloor}-x_{\left \lfloor \frac{N}{2} \right \rfloor-1}\\
\end{bmatrix}
=2~{}^tW_N^{-1}D_N^{-1}W_N^{-1}
\begin{bmatrix}
N-1\\
N-3\\
\vdots\\
N-2l+1\\
\vdots\\
N-1-2 (\left \lfloor \frac{N}{2} \right \rfloor -1)\\
\end{bmatrix}.
\end{equation*}

Again, let us define the new variable vector $\vec{z}:={}^tW_N\vec{y}$.
Then we have
\[
\vec{z}=
\begin{bmatrix}
z_1\\
z_2\\
\vdots \\
z_{\left \lfloor \frac{N}{2} \right \rfloor-1} \\
z_{\left \lfloor \frac{N}{2} \right \rfloor}\\
\end{bmatrix}
=2D_N^{-1}W_N^{-1}
\begin{bmatrix}
N-1\\
N-3\\
\vdots\\
N-1-2(\left \lfloor \frac{N}{2} \right \rfloor -2)\\
N-1-2(\left \lfloor \frac{N}{2} \right \rfloor -1)\\
\end{bmatrix}. 
\]
By solving this matrix equation, we obtain the following. 
\begin{lem}\label{2.2}
$$1\leq \forall l \leq \left \lfloor \frac{N}{2} \right \rfloor-1,\quad
z_l=\frac{2}{5}\left(N-2l+1+\frac{(N-2l-1)F_{2l-1}+2N(-1)^{l-1}}{F_{2l+1}}\right).$$
$$z_{\left \lfloor \frac{N}{2} \right \rfloor}=\frac{2}{5F_N} \left( \left(N-2\left \lfloor \frac{N}{2} \right \rfloor \right) \left(F_{2\left \lfloor \frac{N}{2} \right \rfloor+1}+F_{2\left \lfloor \frac{N}{2} \right \rfloor-1}\right)+F_{2\left \lfloor \frac{N}{2} \right \rfloor}+2N(-1)^{\left \lfloor \frac{N}{2} \right \rfloor-1} \right).$$
\end{lem}

\noindent
{\it Proof of Lemma \ref{2.2}.}

For the case that $1\leq \forall l \leq \left \lfloor \frac{N}{2} \right \rfloor-1$, we have
\begin{align*}
z_l
&=
2D_N^{-1}W_N^{-1}{}^t
\begin{bmatrix}
N-1, N-3, \cdots, N-1-2(\left \lfloor \frac{N}{2} \right \rfloor -1)
\end{bmatrix} \\
&=
\frac{2}{F_{2l+1}}
\sum_{k=1}^{l}(-1)^{l-k}(N-2k+1)F_{2k-1} \\
&=
\frac{2}{F_{2l+1}}
\left((N+1)\sum_{k=1}^{l}(-1)^{l-k}F_{2k-1}-2\sum_{k=1}^{l}kF_{2k-1} \right) \\
&=
\frac{2}{F_{2l+1}}
\left( (N+1)F_{l}^2-\frac{2}{5}F_{2l-1}+\frac{2}{5}(-1)^{l-1}(2l-1)-2lF_{l}^2 \right)&&\text{(by (\ref{t10}))}\\
&=
\frac{2}{5F_{2l+1}}
\left( 5(N-2l+1)F_{l}^2-2F_{2l-1}+2(-1)^{l-1}(2l-1)\right)\\
&=
\frac{2}{5F_{2l+1}}
\left( (N-2l+1)(F_{2l-1}+F_{2l+1}+2(-1)^{l-1})-2F_{2l-1}+2(-1)^{l-1}(2l-1) \right)&&\text{(by (\ref{t11}))}\\
&=
\frac{2}{5F_{2l+1}}
\left( (N-2l-1)F_{2l-1}+(N-2l+1)F_{2l+1}+2N(-1)^{l-1} \right) \\
&=
\frac{2}{5}
\left( (N-2l+1)+\frac{(N-2l-1)F_{2l-1}+2N(-1)^{l-1}}{F_{2l+1}} \right). 
\end{align*}


In the case of $z_{\lfloor \frac{N}{2} \rfloor}$, we have 
\begin{align*}
z_{\lfloor \frac{N}{2}\rfloor}
&=\frac{2}{F_N}\sum_{k=1}^{\lfloor \frac{N}{2}\rfloor}(-1)^{\lfloor \frac{N}{2}\rfloor-k}(N-2k+1)F_{2k-1}\\
&=\frac{2}{F_N}\left(\sum_{k=1}^{\lfloor \frac{N}{2}\rfloor}(-1)^{\lfloor \frac{N}{2}\rfloor-k}F_{2k-1}-2\sum_{k=1}^{\lfloor \frac{N}{2}\rfloor}kF_{2k-1}\right)\\
&=\frac{2}{F_N}\left((N+1)F_{\lfloor \frac{N}{2}\rfloor}^2-\frac{2}{5}F_{2\lfloor \frac{N}{2}\rfloor-1}+\frac{2}{5}(-1)^{\lfloor \frac{N}{2}\rfloor-1}\left(2\lfloor \frac{N}{2}\rfloor-1\right)-2\lfloor \frac{N}{2}\rfloor F_{\lfloor \frac{N}{2}\rfloor}^2\right) &&\text{(by (\ref{t10}))}\\
&=\frac{2}{5F_N}\left(5(N-2\lfloor \frac{N}{2}\rfloor +1)F_{\lfloor \frac{N}{2}\rfloor}^2-2F_{2\lfloor \frac{N}{2}\rfloor-1}+2(-1)^{\lfloor \frac{N}{2}\rfloor-1}\left(2\lfloor \frac{N}{2}\rfloor-1\right)\right)\\
&=\frac{2}{5F_N}\Biggl(\left(N-2\lfloor \frac{N}{2}\rfloor+1\right)\left(F_{2\lfloor \frac{N}{2}\rfloor-1}+F_{2\lfloor \frac{N}{2}\rfloor+1}+2(-1)^{\lfloor \frac{N}{2}\rfloor-1}\right) &&\text{(by (\ref{t11}))}\\
&-2F_{2\lfloor \frac{N}{2}\rfloor-1}+(-1)^{\lfloor \frac{N}{2}\rfloor-1}\left(4\lfloor \frac{N}{2}\rfloor-2\right)\Biggr) \\
&=\frac{2}{5F_N}\left(\left(N-2\lfloor \frac{N}{2}\rfloor\right)\left(F_{2\lfloor \frac{N}{2}\rfloor-1}+F_{2\lfloor \frac{N}{2}\rfloor+1}\right)+F_{2\lfloor \frac{N}{2}\rfloor}+(-1)^{\lfloor \frac{N}{2}\rfloor-1}2N\right).
\end{align*}
\qed

Combining the above lemma with the equation $\vec{y}={}^tW_N^{-1}\vec{z}$, we have the following. 

\begin{lem}\label{2.3}~\\
\hspace{100pt}$\displaystyle y_l=\frac{2}{5}\left((N-2l+1)+2N\frac{(-1)^{l-1}F_{N-2l+1}}{F_N}\right)$. 
\end{lem}

\noindent
{\it Proof of Lemma  \ref{2.3}.}

\begin{align*}
y_{l}
&=F_{2l-1}\sum_{i=l}^{\left \lfloor \frac{N}{2} \right \rfloor -1} (-1)^{l+i}\frac{z_i}{F_{2i-1}}
+(-1)^{l+\left \lfloor \frac{N}{2} \right \rfloor}\frac{F_{2l-1}}{F_{2\left \lfloor \frac{N}{2} \right \rfloor -1}} z_{\left \lfloor \frac{N}{2} \right \rfloor}\\
&=\frac{2}{5}F_{2l-1}\Biggl(\sum_{i=l}^{\left \lfloor \frac{N}{2} \right \rfloor-1}(-1)^{l+i}\frac{(N-2i+1)F_{2i+1}+(N-2i-1)F_{2i-1}+2N(-1)^{i-1}}{F_{2i-1}F_{2i+1}} \\
&+\frac{(-1)^{l+\left \lfloor \frac{N}{2} \right \rfloor}}{F_{N}F_{2\left \lfloor \frac{N}{2} \right \rfloor -1}}\left(\left(N-2\left \lfloor \frac{N}{2} \right \rfloor+1\right)F_{2\left \lfloor \frac{N}{2} \right \rfloor+1}+\left(N-2\left \lfloor \frac{N}{2} \right \rfloor-1\right)F_{2\left \lfloor \frac{N}{2} \right \rfloor-1}+2N (-1)^{\left \lfloor \frac{N}{2} \right \rfloor-1}\right)\Biggr) \\
&=\frac{2}{5}F_{2l-1}\Biggl(\sum_{i=l}^{\left \lfloor \frac{N}{2} \right \rfloor-1}(-1)^{l+i}\left(\frac{(N-2i-1)(F_{2i+1}+F_{2i-1})+2F_{2i+1}-2N(-1)^{i}}{F_{2i-1}F_{2i+1}}\right) \\
&+\frac{(-1)^{l+\left \lfloor \frac{N}{2} \right \rfloor}}{F_{N}F_{2\left \lfloor \frac{N}{2} \right \rfloor-1}}\left((N-2\left \lfloor \frac{N}{2} \right \rfloor+1)F_{2\left \lfloor \frac{N}{2} \right \rfloor+1}+(N-2\left \lfloor \frac{N}{2} \right \rfloor-1)F_{2\left \lfloor \frac{N}{2} \right \rfloor-1}+2N (-1)^{\left \lfloor \frac{N}{2} \right \rfloor-1})\right)\Biggr)\\
&=\frac{2}{5}F_{2l-1}\left(\frac{N-2l+1}{F_{2l-1}}+2N(-1)^{l-1}\sum_{i=l}^{\left \lfloor \frac{N}{2} \right \rfloor-1}\frac{1}{F_{2i-1}F_{2i+1}}+\frac{2N(-1)^{l-1}}{F_{N}F_{2\left \lfloor \frac{N}{2} \right \rfloor-1}}\right)\\
\end{align*}
\begin{align*}
&=\frac{2}{5}F_{2l-1}\left(\frac{N-2l+1}{F_{2l-1}}+2N(-1)^{l-1}\sum_{i=l}^{\left \lfloor \frac{N}{2} \right \rfloor-1}{\left(\frac{F_{2i+2}}{F_{2i+1}}-\frac{F_{2i}}{F_{2i-1}}\right)}+\frac{2N(-1)^{l-1}}{F_{N}F_{2\left \lfloor \frac{N}{2} \right \rfloor-1}}\right) &&\text{(by (\ref{t7}))}\\
&=\frac{2}{5}F_{2l-1}\left(\frac{N-2l+1}{F_{2l-1}}+2N(-1)^{l-1}\left(\frac{F_{2\left \lfloor \frac{N}{2} \right \rfloor}}{F_{2\left \lfloor \frac{N}{2} \right \rfloor-1}}-\frac{F_{2l}}{F_{2l-1}}\right)+2N(-1)^{l-1}\left(\frac{F_{N+1}}{F_{N}}-\frac{F_{2\left \lfloor \frac{N}{2} \right \rfloor}}{F_{2\left \lfloor \frac{N}{2} \right \rfloor-1}}\right)\right) \\
&=\frac{2}{5}\left(N-2l+1+2N(-1)^{l-1}\left(\frac{-F_{N}F_{2l}+F_{N+1}F_{2l-1}}{F_{N}}\right)\right) \\
&=\frac{2}{5}\left(N-2l+1+2N\frac{(-1)^{l-1}F_{N-2l+1}}{F_{N}}\right). &&\text{(by (\ref{t5}))}
\end{align*}\qed

Combining Lemmas  \ref{2.2} and  \ref{2.3} and the equation $x_l=\sum_{i=1}^ly_i$, we obtain the exact formula for the variables $x_l$ as follows:
\begin{align*}
h_N(0,l)=x_{l}
&=\sum_{i=1}^{l} \frac{2}{5}\left((N-2i+1)+2N\frac{(-1)^{i-1}F_{N-2i+1}}{F_{N}}\right) \\
&=\frac{2}{5}l(N+1)-\frac{4}{5}\sum_{i=1}^{l}i-\frac{4N}{5F_{N}}\sum_{i=1}^{l}(-1)^{i}(F_{N-1}F_{2i-1}-F_{N}F_{2i-2}) \\
&=\frac{2}{5}l(N-l)-\frac{4N}{5F_{N}}\left(F_{N-1}\sum_{i=1}^{l}(-1)^{i}F_{2i-1}-F_{N}\sum_{i=1}^{l}(-1)^{i}F_{2i-2}\right) \\
&=\frac{2}{5}l(N-l)+\frac{4N}{5F_{N}}\left((-1)^{l-1}F_{N-1}F_{l}^{2}+(-1)^{l}F_{N}F_{l-1}F_{l}\right) &&\text{(by (\ref{t10}))}\\
&=\frac{2}{5}l(N-l)+\frac{4N}{5F_{N}}F_{l}F_{N-l}=
\frac{2}{5}\left(l(N-l)+2N\frac{F_{l}F_{N-l}}{F_{N}}\right). \hspace{45pt} \Box
\end{align*}

From Theorem \ref{decomposition} we have $H_N^{-1}={}^tU_N\,{}^tW^{-1}_ND^{-1}_NW^{-1}_N\,U_N$. Hence
the entries of $H^{-1}_N$ also can be calculated as follows:

\noindent
For $1\leq i,j\leq \left \lfloor \frac{N}{2} \right \rfloor$,
\begin{align*}
H_N^{-1}(i,j)=
\begin{cases}
\displaystyle\frac{F_iF_jF_{N-i}F_{N-j}}{F_NF_{N-1}}+\displaystyle\sum_{k=1}^{i-1}\frac{F_{i-k}F_{j-k}F_{N-i-k}F_{N-j-k}}{F_{N-2k+1}F_{N-2k-1}} &\text{if $i\leq j$,}\\
\displaystyle\frac{F_iF_jF_{N-i}F_{N-j}}{F_NF_{N-1}}+\displaystyle\sum_{k=1}^{j-1}\frac{F_{i-k}F_{j-k}F_{N-i-k}F_{N-j-k}}{F_{N-2k+1}F_{N-2k-1}} &\text{otherwise.}
\end{cases}
\end{align*}

Combining the above with the equation $H_N \vec{x}=$$4\vec{1}$, we have the following.

\begin{prop} \label{lem3.1}
\begin{align*}
h_N(0,1)=\frac{2}{F_N}\sum_{i=0}^NF_iF_{N-i}. 
\end{align*}
\end{prop}

\begin{proof}
\begin{align*}
h_N(0,1)&=4\sum_{j=1}^{\left \lfloor \frac{N}{2} \right \rfloor}H_N^{-1}(1,j)
=4\sum_{j=1}^{\left \lfloor \frac{N}{2} \right \rfloor}\frac{F_1F_jF_{N-1}F_{N-j}}{F_NF_{N-1}}
=4\sum_{j=1}^{\left \lfloor \frac{N}{2} \right \rfloor}\frac{F_jF_{N-j}}{F_N}
=\frac{2}{F_N}\sum_{j=0}^NF_jF_{N-j}. 
\end{align*}
\end{proof}

\begin{rem}
We also obtain the counterpart of Theorem \ref{decomposition} for the case of $C^3_N$ \cite{C3N},
which convinced us that our method is always efficient for the case of general 
$C^m_N$, although their calculations get much complicated as the parameter $m$ get larger. 
\end{rem}

\section{Effective Resistance, Kirchhoff Index, and Graph Complexity}

The {\it effective resistance} between a pair of vertices $\{p,q\}$ of
a connected graph $G$, denoted by $r(G;p,q)$, is the electrical resistance
measured across $p$ and $q$ when the graph $G$ represents an electrical
circuit such that each edge of $G$ is a unit resistor.

In 1959, Nash-Williams \cite{Nash} proved the following surprising formula
between the effective resistances and the average hitting times. 
\begin{thm}[Nash-Williams, 1959 \cite{Nash}] \label{Nash}
\begin{align*}
r(G;x,y)=\frac{h(G;x,y)+h(G;y,x)}{2|E(G)|}.
\end{align*}
\end{thm}

The Kirchhoff index $Kf(G)$ of the graph $G$ is defined as 
$Kf(G):=\sum_{i<j}r(G;v_i,v_j)$.

\begin{thm}\label{ki}
\begin{align*}
\mathrm{Kf}(C_N^2)=\frac{N(N-1)(5N+17)}{300}+\frac{2N^2}{25}\frac{F_{N-1}}{F_N}. 
\end{align*}
\end{thm}

\begin{proof}
The following equation is derived immediately from Theorem \ref{main} and Proposition \ref{lem3.1}.  
\begin{align} \label{fib3.2}
\sum_{i=0}^NF_iF_{N-i}=\frac{1}{5}((N-1)F_N+2NF_{N-1}).
\end{align}
Then we have:
\begin{align*}
\mathrm{Kf}(C_N^2)
&=\sum_{x<y}r(C_N^2;x,y)
=\sum_{i=1}^{N-1}(N-i)r(C_N^2;0,i)
=\sum_{i=1}^{N-1}(N-i)\frac{h(C_N^2;0,i)+h(C_N^2;i,0)}{2|E(C_N^2)|}\\
&=\sum_{i=1}^{N-1}(N-i)\frac{h(C_N^2;0,i)}{2N}
=\sum_{i=1}^{N-1}\frac{(N-i)}{5N}\left(i(N-i)+2N\frac{F_iF_{N-i}}{F_N}\right)\\
&=\frac{1}{5}\left(N\sum_{i=1}^{N-1}i-\sum_{i=1}^{N-1}i^2+\frac{2N}{F_N}\sum_{i=1}^{N-1}F_iF_{N-i}\right)-\frac{1}{5N}\left(N\sum_{i=1}^{N-1}i^2-\sum_{i=1}^{N-1}i^3+\frac{2N}{F_N}\sum_{i=1}^{N-1}iF_iF_{N-i}\right)\\
&=\frac{1}{5}\left(N\sum_{i=1}^{N-1}i-2\sum_{i=1}^{N-1}i^2+\frac{2N}{F_N}\sum_{i=1}^{N-1}F_iF_{N-i}\right)-\frac{1}{5N}\left(-\sum_{i=1}^{N-1}i^3+\frac{N^2}{F_N}\sum_{i=1}^{N-1}F_iF_{N-i}\right)\\
&=\frac{1}{5}\left(N\sum_{i=1}^{N-1}i-2\sum_{i=1}^{N-1}i^2+\frac{1}{N}\sum_{i=1}^{N-1}i^3+\frac{N}{F_N}\sum_{i=1}^{N-1}F_iF_{N-i}\right)\\
&=\frac{1}{5}\left(\frac{N^2(N-1)}{2}-\frac{N(N-1)(2N-1)}{3}+\frac{N(N-1)^2}{4}+\frac{N}{5F_N}((N-1)F_N+2NF_{N-1})\right) &&\text{(by (\ref{fib3.2}))}\\
\end{align*}
\begin{align*}
&=\frac{1}{5}\left(\frac{N(N-1)(N+1)}{12}+\frac{N}{5F_N}(2NF_{N-1}+(N-1)F_N)\right)\\
&=\frac{1}{5}\left(\frac{N(N-1)(N+1)}{12}+\frac{2N^2F_{N-1}}{5F_N}+\frac{N(N-1)}{5}\right)\\
&=\frac{1}{5}\left(\frac{N(N-1)(5N+17)}{60}+\frac{2N^2F_{N-1}}{5F_N}\right)\\
&=\frac{N(N-1)(5N+17)}{300}+\frac{2N^2}{25}\frac{F_{N-1}}{F_N}. 
\end{align*}
\end{proof}

The {\textit{graph complexity}} of a graph $G$, denoted by $t(G)$,
means the number of spanning trees in $G$. 
Then let $t(G;x,y)$ denote the graph complexity of the graph
obtained by identifying two vertices $x$ and $y$ of $G$. 
The following famous formula between the effective resistances and
the graph complexities was proven in Kirchhoff \cite{Kirchhoff}.   
\begin{thm}[G. Kirchhoff, 1847 \cite{Kirchhoff}]\label{Kirchhoff}
\begin{align*}
r(G;x,y)=\frac{t(G;x,y)}{t(G)}.
\end{align*}
\end{thm}

Furthermore, D. J. Kleitman and B. Golden \cite{Kleitman} proved
the following beautiful formula for the graph complexity of $C_{N}^{2}$.  
\begin{thm}[D. J. Kleitman and B. Golden, 1975 \cite{Kleitman}]\label{Kleitman}
\begin{align*}
t(C_{N}^{2})=NF_N^2.
\end{align*}
\end{thm}
Combining Theorems \ref{main}, \ref{Nash}, \ref{Kirchhoff}, and \ref{Kleitman},
we have the following.
\begin{prop}
\begin{align*}
t(C_N^2;0,l)=\frac{F_{N}}{5}\left(l(N-l)F_N+2NF_{l}F_{N-l}\right).
\end{align*}
\end{prop}

\begin{proof}
\begin{align*}
t(C_N^2;0,l)&=r(C_N^2;0,l)t(C_N^2)
=\frac{h_N(0,l)F_N^2}{2}
=\frac{F_N^2}{5}\left(l(N-l)+2N\frac{F_{l}F_{N-l}}{F_N}\right)\\
&=\frac{F_{N}}{5}\left(l(N-l)F_N+2NF_{l}F_{N-l}\right). 
\end{align*}
\end{proof}

\section{Asymptotic Behavior}

Theorem $1.2$ is useful for calculating the asymptotic behavior of HT's of simple
random walks on $C_N^2$. 
Let $\phi=\frac{1+\sqrt5}{2}$. 

\begin{lem}[\cite{Simon}]
\begin{align*}
\lim_{n \to \infty} \frac{F_n}{F_{n-1}}=\phi. 
\end{align*}
\label{Wall}
\end{lem}

\begin{thm}
\begin{align*}
\lim_{N \to \infty}\frac{h_N(0,l)-\frac{2}{5}l(N-l)}{N}=\frac{4}{5\sqrt5} \qquad(l=0, 1, 2, \cdots). 
\end{align*}
\end{thm}

\begin{proof}
By Theorem \ref{main} and Lemma \ref{Wall}, we have
\begin{align*}
h_N(0,l)
&=\frac{2}{5}\left(l(N-l)+2N\frac{F_l \cdot F_{N-l}}{F_N} \right) \notag \\
&=\frac{2}{5}\left(l(N-l)+2N\frac{1}{\sqrt5} \left(\frac{\phi^l \cdot \phi^{N-l}}{\phi^N}+O\left(\frac{1}{N}\right)\right)\right) \notag \\
&=\frac{2}{5}\left(l(N-l)+\frac{2N}{\sqrt5}\left(1+O\left(\frac{1}{N}\right)\right)\right). 
\end{align*}
Then,
\begin{align*}
\lim_{N \to \infty}\frac{h_N(0,l)-\frac{2}{5}l(N-l)}{N}=\lim_{N \to \infty}\frac{4}{5\sqrt5}\left(1+O\left(\frac{1}{N}\right)\right)=\frac{4}{5\sqrt5}.
\end{align*}
\end{proof}

\begin{thm}
\begin{align*}
\lim_{\substack{N \to \infty \\ \frac{l}{N} \to x}}\frac{h_N(0,l)}{N^2}=\frac{2}{5}x(1-x). 
\end{align*}
\end{thm}

\begin{proof}
\begin{align*}
\lim_{\substack{N \to \infty \\ \frac{l}{N} \to x}}\frac{h_N(0,l)}{N^2}
&=\lim_{\substack{N \to \infty \\ \frac{l}{N} \to x}}\frac{2}{5}\left(\frac{l}{N}\left(1-\frac{l}{N}\right)+\frac{2}{N} \cdot \frac{1}{\sqrt5}\left(1+O\left(\frac{1}{N}\right)\right)\right)\\
&=\frac{2}{5}x(1-x). 
\end{align*}
\end{proof}


\end{document}